\documentclass{amsart}

\usepackage{amssymb,color}
\usepackage{amsfonts}
\usepackage{amsmath}
\usepackage{euscript}
\usepackage{enumerate}
\usepackage{graphics}

\usepackage{hyperref}
\usepackage{pdfsync}
\newtheorem{theorem}{Theorem}[section]

\newtheorem{prop}[theorem]{Proposition}

\newtheorem*{Theorem1'}{Theorem 1'}

\theoremstyle{definition}

\theoremstyle{remark}
\newtheorem{remark}[theorem]{Remark}

\numberwithin{equation}{section}

\newcommand \blue{\textcolor{blue}}

\newcommand \Z{{\mathbb Z}}

\newcommand \C{{\mathbb C}}

\newcommand \Hom{{\mathrm {Hom}}}

\newcommand \gl{{\mathfrak {gl}}}
\newcommand \sel{{\mathfrak {sl}}}
\renewcommand \r{{\mathfrak {r}}}
\newcommand \g{{\mathfrak {g}}}

\newcommand \n{{\mathfrak {n}}}
\newcommand \h{{\mathfrak {h}}}

\newcommand \s{{\mathfrak {s}}}

\renewcommand \sl{{\mathfrak {sl}}}
\newcommand \z{{\mathfrak z}}

\begin{document}

\title[Uniserial representations of
conformal Galilei algebras]{
Classification of  finite dimensional uniserial representations of
conformal Galilei algebras}

\author{Leandro Cagliero}
\address{CIEM-CONICET, FAMAF-Universidad Nacional de C\'ordoba, Argentina.}
\email{cagliero@famaf.unc.edu.ar}

\author{Luis Guti\'errez Frez}
\address{Instituto de Ciencias F\'{\i}sicas y Matem\'aticas, Universidad Austral de Chile, Chile}
\email{ luis.gutierrez@uach.cl}

\author{Fernando Szechtman}
\address{Department of Mathematics and Statistics, Univeristy of Regina, Canada}
\email{fernando.szechtman@gmail.com}





\begin{abstract} With the aid of the $6j$-symbol, we classify all uniserial modules
of $\sl(2)\ltimes \h_{n}$, where $\h_{n}$ is the Heisenberg Lie algebra of dimension $2n+1$.
\end{abstract}

\maketitle

\section{Introduction}
\label{intro}

We fix throughout a field $\mathbb{F}$ of characteristic zero. All Lie algebras and representations considered in this paper are
assumed to be finite dimensional over $\mathbb{F}$, unless explicitly stated otherwise.

Given a Lie algebra $\g$ and a $\g$-module $V$, the \emph{socle series} of $V$, namely
\[
0=\mathrm{soc}^{0}(V)\subset\mathrm{soc}^{1}(V)\subset
\cdots\subset \mathrm{soc}^{m}(V)=V
\]
is inductively defined by
declaring $\mathrm{soc}^{i}(V)/\mathrm{soc}^{i-1}(V)$ to be the socle
of $V/\mathrm{soc}^{i-1}(V)$, that is, the sum of all irreducible
submodules of $V/\mathrm{soc}^{i-1}(V)$,  $1\leq i\leq m$.
By definition, $V$ is \emph{uniserial}  if
the \emph{socle factors} $\mathrm{soc}^{i}(V)/\mathrm{soc}^{i-1}(V)$ are
irreducible for all $1\leq i\leq m$.
In other words, $V$ is uniserial if its socle series is a composition
series, or equivalently if its submodules are totally ordered by inclusion.

Uniserial and serial modules or rings
are very important in the context of  associative algebras and there is
an extensive literature devoted to them.
For instance, the class of serial rings and algebras includes discrete
valuation rings, Nakayama algebras, triangular matrix rings over a skew field and
Artinian principal ideal rings (see \cite{EG,Pu}). In particular, every proper factor ring
of a Dedekind domain is serial. Also, serial algebras occur as the group algebras in
characteristic $p$ (see, for instance \cite{Sr}).
In \cite{BH-Z}, among other things, a
characterization of algebras of finite uniserial type is given.
In contrast, there are only few papers devoted to the study these concepts for Lie algebras.

This work is a new step in a project aiming to systematically investigate the
uniserial representations of Lie algebras.
Here, we  classify all uniserial $\g$-modules
for $\g=\sl(2)\ltimes \h_{n}$, where $\h_{n}$ is the Heisenberg Lie algebra of dimension $2n+1$ and $\sl(2)$ acts on $\h_{n}$
so that both the center $\z$ of $\h_{n}$ and $\h_{n}/\z$
are irreducible $\sl(2)$-modules.
More precisely, given an integer $a\geq 0$, let $V(a)$ be the irreducible $\sl(2)$-module with highest weight $a$.
Thus,
\[
 \h_n\simeq V(m)\oplus V(0),\quad m=2n-1,
\]
as $\sl(2)$-modules.
The Lie algebra $\g$ is a conformal Galilei algebra and it is an important object in mathematical physics.
Galilei algebras and their representations attract considerable attention (see \cite{AIK}, \cite{LMZ} and references therein).

Previously, we obtained a classification of all uniserial $\g$-modules when
$\g=\sl(2)\ltimes V(a)$, $a\geq 1$, over the complex numbers (see \cite{CS}), as well as when $\g$ is abelian, over a sufficiently large perfect field (see \cite{CS_Can}). In the first case the classification turns out to be  equivalent to determining all non-trivial
zeros of the Racah-Wigner $6j$-symbol within certain parameters, while in the second a sharpened version of the Primitive Element Theorem plays a central role, specially over finite fields.

\medskip

In this article we focus on $\g=\sl(2)\ltimes \h_{n}$.
Since every non-trivial ideal of $\g$ contains $\z$, it follows that any non-faithful
representation of $\g$ is obtained from a representation of $\sl(2)\ltimes V(m)$.
Therefore, the classification of all non-faithful uniserial $\g$-modules follows
from \cite{CS}, while the classification of all faithful uniserial $\g$-modules is
given by the following theorem, which is the main result of the paper.

\begin{theorem}\label{thm.main}
All faithful uniserial $\g$-modules have length 3. Moreover,
the socle factors of a faithful uniserial $\g$-module of length 3 must be one of the following:

\medskip

\noindent
\begin{tabular}{ll}
$m=1:$   & $V(a),V(a+1),V(a)$ or $V(a+1),V(a),V(a+1)$, with $a\ge0$. \\[1mm]
$m=3:$   &  $V(0),V(3),V(0)$ or $V(1),V(4),V(1)$ or $V(1),V(2),V(1)$ or $V(4),V(3),V(4)$.\\[1mm]
$m\geq 5:$   &  $V(0),V(m),V(0)$ or $V(1),V(m+1),V(1)$ or $V(1),V(m-1),V(1)$.

\end{tabular}

\medskip

\noindent
Furthermore, each of these sequences arises from one and only one
isomorphism class of uniserial $\g$-modules.
\end{theorem}

\begin{remark}
It follows from this theorem that for a given $n>2$, $\sl(2)\ltimes\h_{n}$  has only 3
isomorphism classes of faithful uniserial representations (if $n = 2$, it has 4),
whereas it has infinitely many classes of non-faithful ones.
\end{remark}

Theorem \ref{thm.main} is a direct consequence of
Theorems \ref{uno}, \ref{dos} and \ref{no5} below.
Explicit realizations of these modules are given in \S\ref{sec.construction}.

Let us say a few words about Theorem \ref{thm.main}. Suppose $\g=\s\ltimes\n$, with $\s$ simple, $\n$ nilpotent, and
assume that $\n$ is generated as a Lie algebra by an $\s$-submodule $\n_0\subset\n$.
By general results of the theory, in order to obtain a faithful uniserial $\g$-module of length $\ell$
with socle factors $V_i$, $1\le i\le\ell$, the following is required:
\begin{enumerate}[(1)]
 \item for $1\le i\le\ell$, a matrix representation $R_i:\s\to\gl(d_i)$ corresponding to the $\s$-module $V_i$; at least one $R_i$ must be non-trivial.
 \item for $2\le i\le\ell$, a faithful matrix presentation  $X_i:\n_0\to M_{d_{i-1},d_i}(\mathbb{F})$ corresponding to
an $\s$-module homomorphism $\n_0\to \text{Hom}(V_i,V_{i-1})$;
\item for the linear map $R:\s\oplus\n_0\to\gl(\sum d_i)$, defined by
$$
R(s+u)=\left(
         \begin{matrix}
           R_1(s) & X_2(u) & 0 & \cdots & 0\\
           0 & R_2(s) & X_3(u) & \cdots & 0 \\
           \vdots &  & \ddots   &   & \vdots\\
           0 & 0 &  \cdots &   R_{\ell-1}(s) & X_{\ell}(u)  \\
           0 & 0 & \cdots  & 0& R_\ell(s)\\
         \end{matrix}
       \right),\quad s\in\s,\; u\in \n_0,
$$
the matrix Lie algebra $\tilde \n$ generated by $R(\n_0)$ must be isomorphic to $\n$ (note that $\tilde \n$ consists of block upper triangular matrices).
\end{enumerate}

\medskip

When $\g=\sl(2)\ltimes \n$, the following fact describes what happens with
the second superdiagonal of $\tilde \n$, namely $[R(\n_0),R(\n_0)]$, which should be isomorphic to $[\n_0,\n_0]\subset\n$
as $\s$-modules.

\medskip
\noindent
\textbf{Fact:} Generically speaking, it turns out that the block $(i-2,i)$, $3\le i\le\ell$,
of $\tilde \n$ consists of \textbf{all} irreducible $\s$-submodules of $\Lambda^2\n_0$ that also appear in $\text{Hom}(V_i,V_{i-2})$.
Nevertheless, in some curious cases, some of these constituents do not appear in $[R(\n_0),R(\n_0)]$,
as the following example shows.

\medskip
\noindent
\textbf{Example:} Let $\g=\sl(2)\ltimes\h_2$, so that $\n_0=V(3)$ and $[\n_0,\n_0]=V(0)$.
Assume that $\{v_0,v_1,v_2,v_3\}$ is \blue{a} standard basis of $\n_0$, as defined in \S\ref{wdos}.
Proceeding as above, we obtain a faithful representation of $\g$ with socle factors $V(4)$, $V(3)$, $V(4)$ via:
\[
 R\big(\textstyle\sum_{i=0}^3a_iv_i\big)=\left(
         \begin{array}{c|c|c}
 \qquad 0\qquad  & \begin{smallmatrix}
               -6a_1&6a_0&0&0\\ -3a_2&0&3a_0&0\\ -a_3&-3a_2&3a_1&a_0\\ 0&-3a_3&0&3a_1\\ 0&0&-6a_3&6a_2
               \end{smallmatrix}
                                 &  \\[5mm] \hline
            & 0 &\begin{smallmatrix}
                    \\[1mm] 3a_2&-6a_1&3a_0&0&0\\ a_3&0&-3a_1&2a_0&0\\ 0&2a_3&-3a_2&0&a_0\\ 0&0&3a_3&-6a_2&3a_1
                  \end{smallmatrix} \\[5mm] \hline
                 & \begin{smallmatrix}\\[6mm]  \end{smallmatrix}    & 0
         \end{array}
       \right).
\]
It turns out that, by ``some miracle'', $[R(\n_0),R(\n_0)]$ is just $V(0)$, as opposed to the expected result
of $V(0)\oplus V(4)$ (note that $V(4)$ is indeed a constituent of both $\text{Hom}(V(4),V(4))$
and $\Lambda^2\n_0$. This ``miracle'', which is due to the exceptional zero
$\left\{\begin{matrix}
\frac{4}{2} \; \frac{3}{2} \; \frac{3}{2}  \\[.6mm]
\frac{3}{2} \; \frac{4}{2} \; \frac{3}{2}
\end{matrix}\right\}=0$
of the $6j$-symbol, produces an unexpected uniserial $\g$-module.

\medskip

In general, if $\g=\sl(2)\ltimes \n$ then the exceptional
zeros of the $6j$-symbol control when the matrix Lie algebra generated by $R(\n_0)$ is isomorphic to $\n$.
Item (3) above might be very difficult to determine for other simple Lie algebras $\s$.

\section{Preliminaries}

\subsection{Matrix recognition of uniserial representations}\label{sec.matrix_recognition}
In this subsection we recall from \cite{CS}
some basic facts about uniserial representations
of a Lie algebra $\g$ with solvable radical $\r$ and
fixed Levi decomposition $\g=\s\ltimes \r$.

Given a representation $T:\g\to\gl(V)$ and a
basis $B$ of $V$ we let $M_B:\g\to\gl(d)$, $d=\dim(V)$, stand for
the corresponding matrix representation.

Since  $\s$ is semisimple, it follows that there exist
irreducible $\s$-submodules $W_i$, $1\le i\le n$, such that
\begin{equation}\label{eq.comp_series}
0\subset W_1\subset W_1\oplus W_2\subset W_1\oplus
W_2\oplus W_3\subset \cdots\subset W_1\oplus\cdots \oplus
W_n=V
\end{equation}
is the composition series of $V$.
Let $B=B_1\cup\cdots\cup B_n$ be a basis of $V$, where each $B_i$ is a basis of $W_i$.
We say that $B$ is \emph{adapted} to the composition series \eqref{eq.comp_series}.
If $B$ is adapted to a composition series, then $M_B(s)$ is block diagonal for all $s\in\s$.

It is well-known \cite[Ch. I, \S 5, Th. 1]{Bo} that $[\g,\r]$ annihilates every irreducible $\g$-module. Therefore, if $B$ is adapted to a composition series, then
$M_B(r)$ is strictly block upper triangular for all $r\in [\g,\r]$.

The following result, proven in \cite[Theorem 2.4]{CS} over $\C$, remains valid over $\mathbb{F}$.

\begin{theorem}\label{thm.CS_2.4}
\label{fed} The $\g$-module $V$ is uniserial if and only if, given
any basis $B$ adapted to a composition series,
none of the blocks in the first
superdiagonal of $M_B(\r)$ is identically 0.
Moreover, if $[\g,\r]=\r$ and there exists a basis $B$ adapted to a composition series
such that none of the blocks in the first
superdiagonal of $M_B(\r)$ is identically 0, then $V$ is
uniserial.
\end{theorem}

\subsection{Uniserial representations of \texorpdfstring{$\sl(2)\ltimes V(m)$}{sl(2)xV(m)}}\label{wdos}
Recall that $V(a)$ is the irreducible $\sl(2)$-module with highest weight $a\ge0$.
We fix a basis $\{v_0,\dots,v_a\}$ of $V(a)$ relative to which $e,h,f\in\sl(2)$ act as follows:
$$
hv_i=(a-2i)v_i,
$$
$$
ev_i=(a-(i-1))v_{i-1},
$$
$$
fv_i=(i+1)v_{i+1},
$$
where $0\leq i\leq a$ and $v_{-1}=0=v_{a+1}$.
Such  basis of $V(a)$ will be called \emph{standard}. We write $R_a:\sl(2)\to\gl(a+1)$ for
the corresponding matrix representation.

The following theorem, proved in \cite{CS}, provides a classification, up to isomorphism,
of all the uniserial representations of the Lie algebra $\sl(2)\ltimes V(m)$, $m\ge1$,
when the underlying field is $\C$. Nevertheless, the classification remain true over $\mathbb{F}$.

\begin{theorem}\label{thm.CS_Classification}
Up to a reversing of the order,
the following are the only possible sequences of socle factors of
uniserial representations of $\sl(2)\ltimes V(m)$:

\noindent
\begin{tabular}{ll}
\\[-2mm]
Length 1. & $V(a)$.  \\[2mm]
Length 2. & $V(a),V(b)$, where $a+b\equiv m\mod 2$ and $0\le b-a\leq m\leq a+b$. \\[2mm]
Length 3. & $V(a),V(a+m),V(a+2m)$; or \\[1mm]
      & $V(0),V(m),V(c)$, where $c\equiv 2m \mod 4$ and $c\leq 2m$. \\[2mm]
Length 4. & $V(a),V(a+m),V(a+2m),V(a+3m)$; or \\[1mm]
      & $V(0),V(m),V(m),V(0)$, where $m\equiv 0\mod 4$. \\[2mm]
Length $\geq 5$. &  $V(a),V(a+m),\dots,V(a+s m)$, where $s\geq 4$. \\[2mm]
\end{tabular}

\noindent Moreover, each of these sequences arises from only one
isomorphism class of uniserial $\g$-modules, except for the
sequence $V(0),V(m),V(m),V(0)$, $m\equiv 0\mod 4$. The isomorphism
classes of uniserial $\g$-modules associated to this sequence are
parametrized by $\mathbb{F}$.
\end{theorem}

Explicit realizations of these modules can be found in \cite{CS}.

\subsection{The Lie algebra \texorpdfstring{$\g=\sl(2)\ltimes \h_{n}$}{}}\label{lieg}
We fix an integer $n\ge1$.
Let $\h_{n}$ be the  Heisenberg Lie algebra of dimension $2n+1$
and set $m=2n-1$.
Of all Lie algebras of a given dimension (that, a fortiori, must be odd),
$\h_{n}$ is characterized by the fact that its center, say $\z=\C z$,
is 1-dimensional and $[\h_n,\h_n]=\z$.

We know that $\sl(2)$ acts via derivations on $\h_{n}$ in such a way that
\[
\h_{n}\simeq V(m)\oplus \z
\]
as $\sl(2)$-modules, where $\z\simeq V(0)$. There is a unique $\sl(2)$-invariant skew-symmetric bilinear form on $V(m)$, up to scaling.
Thus, the bracket on $V(m)$ is uniquely determined, up to scaling. We fix $[v_0,v_m]=z$ and obtain
\[
[v_i,v_{m-i}]=(-1)^i\binom{m}{i}z,\quad 0\leq i\leq m.
\]
Let $\g=\sl(2)\ltimes \h_{n}$.

\subsection{The \texorpdfstring{$6j$}{}-symbol}

Given three half-integers, $j_1$, $j_2$ and $j_3$,
we say that they satisfy the triangle condition if $j_1+j_2+j_3\in\Z$ and
there is a triangle with sides $j_1$, $j_2$ and $j_3$; that is
\[
|j_1-j_2|\le j_3\le j_1+j_2.
\]
In particular,  $j_1$, $j_2$ and $j_3$ must be non-negative.
If either $|j_1-j_2|=j_3$ or $j_3=j_1+j_2$ we say that the triple
 $(j_1,j_2, j_3)$ is a degenerate triangle.
The Clebsch-Gordan formula states that $\dim_{\mathbb{F}}\Hom_{\sl(2)}(V(k), V(a)\otimes V(b))=1$
if $(\frac{a}{2},\frac{b}{2},\frac{k}{2})$ satisfies the triangle condition and 0 otherwise.

\medskip

We recall from \cite[Chapter 9]{VMK} some of the main properties of the $6j$-symbol.

\begin{enumerate}[(1)]
 \item\label{it.2} Given six half-integers  $j_1$, $j_2$, $j_3$, $j_4$, $j_5$ and $j_6$ the $6j$-symbol
$\left\{
\begin{matrix}
j_1 \;  j_2 \;  j_3 \\
j_4 \;  j_5 \;  j_6
\end{matrix}
\right\}$
is a real number that is, by definition, zero if
one of following four triples
\begin{equation}
\label{triples}
 (j_1,j_2, j_3),\; (j_1, j_5, j_6),\; (j_4, j_2, j_6),\; (j_4, j_5, j_3)
 \end{equation}
does not satisfy the triangle condition.
In particular, $\left\{
\begin{matrix}
j_1 \;  j_2 \;  j_3 \\
j_4 \;  j_5 \;  j_6
\end{matrix}
\right\}=0$ if some $j_i<0$.
In contrast, if all four triples (\ref{triples}) satisfy the triangle condition and one
of them is a degenerate triangle
 then $\left\{
\begin{matrix}
j_1 \;  j_2 \;  j_3 \\
j_4 \;  j_5 \;  j_6
\end{matrix}
\right\}\ne0$ (see \cite[\S9.5.2]{VMK}).
\item\label{it.3} The Biedenharn-Elliott identity yields, in particular, the following
three-term recurrence relation (see \cite[\S9.6.2]{VMK} or \cite[pag. 1963]{SG})
\begin{align*}
i_1E(i_1+1)  \left\{\begin{matrix}
i_1\!\!+\!\!1 \; i_2 \; i_3 \\
\;\; i_4 \;\; i_5 \; i_6
\end{matrix}
\right\}
+F(i_1)  \left\{\begin{matrix}
i_1 \; i_2 \; i_3 \\
i_4 \; i_5 \; i_6
\end{matrix}
\right\}
+(i_1+1)E(i_1)  \left\{\begin{matrix}
i_1\!\!-\!\!1 \; i_2 \; i_3 \\
\;\; i_4\;\; i_5 \; i_6
\end{matrix}
\right\}=0
\end{align*}
where
\begin{align*}
F(i_1)=
(2i_1 + 1)\big(&
 i_1(i_1+1)(-i_1(i_1+1) + i_2(i_2+1) + i_3(i_3+1)) \\
 + &i_5(i_5+1)( i_1(i_1+1) + i_2(i_2+1) - i_3(i_3+1)) \\
 + &i_6(i_6+1)( i_1(i_1+1) - i_2(i_2+1) + i_3(i_3+1)) \\
 - &2i_1(i_1+1)i_4(i_4+1)
\big)
\end{align*}
and
\begin{align*}
E(i_1)\!=\!
 \sqrt{\big(i_1^2 - (i_2-i_3)^2\big)\big((i_2+i_3+1)^2 - i_1^2\big)
 \big(i_1^2 - (i_5-i_6)^2\big)\big((i_5+i_6+1)^2 - i_1^2\big)}.
\end{align*}

\item\label{it.4} The $6j$-symbol is invariant under the permutation of any two columns.
It is also invariant if upper and lower arguments are interchanged in any two columns.

\end{enumerate}

\begin{prop}\label{prop.6j_non-zero}
Let   $j_1$, $j_2$, $j_3$, $j_4$, $j_5$ and $j_6$ be
non-negative
half-integers such that $j_1=j_5+j_6\ge3$,
$j_2=j_3$ and all the triples
\begin{equation}\label{eq.4triples}
  (h, j_2, j_3),\; (h, j_5, j_6),\; (j_4, j_2, j_6),\; (j_4, j_5, j_3)
\end{equation}
satisfy the triangle condition for  $h=j_1$, $h=j_1-1$.
If
$
  \left\{\begin{matrix}
j_1\!-\!1 \; j_2 \; j_3 \\
\;\;j_4 \;\; \; j_5 \; j_6
\end{matrix}
\right\}=0
$
then
$
  \left\{\begin{matrix}
j_1\!-\!2 \; j_2 \; j_3 \\
\;\;j_4 \;\; \; j_5 \; j_6
\end{matrix}
\right\}\ne0
$
and
$
  \left\{\begin{matrix}
j_1\!-\!3 \; j_2 \; j_3 \\
\;\;j_4 \;\; \; j_5 \; j_6
\end{matrix}
\right\}\ne0
$.
In particular, the triples \eqref{eq.4triples} satisfy
the triangle condition for  $h=j_1-2$ and $h=j_1-3$.
\end{prop}
\begin{proof}
Fix $(i_2,i_3,i_4,i_5,i_6)=(j_2,j_3,j_4,j_5,j_6)$.
Since $j_2=j_3$, we have
\begin{align*}
E(i_1)\!&=\!
 \sqrt{i_1^2\big((2j_2+1)^2 - i_1^2\big)
 \big(i_1^2 - (j_5-j_6)^2\big)\big((j_5+j_6+1)^2 - i_1^2\big)}. \\
 F(i_1)&=
-(2i_1 + 1)i_1(i_1+1) \\
&\hspace{1cm}\times(
 i_1(i_1+1) - 2j_2(j_2+1) - j_5(j_5+1) - j_6(j_6+1) + 2j_4(j_4+1)
).
\end{align*}
As the triangle condition is satisfied by
$(j_1-1, j_5, j_6)$, we get $|j_5-j_6|\le j_1-1$ and thus  $|j_5-j_6|<j_1$. Likewise, as
the triangle conditions satisfied by
$(j_1, j_2, j_2)$, we get $j_1<2j_2+1$. Moreover, by hypothesis, $j_1=j_5+j_6$, so
$j_1<j_5+j_6+1$.
Recalling that $j_1>0$, these inequalities imply that
\[
E(j_1)\ne0.
\]
Observe next that $F(j_1)=0$. Indeed, $(j_1+1,j_5,j_6)$ does not satisfy the triangle condition and,
by hypothesis, $
  \left\{\begin{matrix}
j_1\!-\!1 \; j_2 \; j_3 \\
\;\;j_4 \;\; \; j_5 \; j_6
\end{matrix}
\right\}=0.
$
It follows from Property \eqref{it.3} applied to $i_1=j_1$ that
\begin{equation}
\label{uo}
F(j_1)  \left\{\begin{matrix}
j_1 \; j_2 \; j_3 \\
j_4 \; j_5 \; j_6
\end{matrix}
\right\}=0.
\end{equation}
But the second factor is non-zero since all four triples (\ref{triples}) taken from (\ref{uo}) satisfy the triangle condition
and $(j_1,j_5,j_6)$ is a degenerate triangle. Thus $F(j_1)=0$.

We next claim that  $F(j_1-2)\ne0$. Indeed, from $j_1>0$ and $F(j_1)=0$ we obtain
\begin{equation*}\label{eq.Cond1}
  j_1(j_1+1) - 2j_2(j_2+1) - j_5(j_5+1) - j_6(j_6+1) + 2j_4(j_4+1)=0.
\end{equation*}
If $F(j_1-2)=0$ then $j_1>2$ implies $j_1(j_1+1) = (j_1-2)(j_1-1)$, so $j_1=\frac{1}{2}$, a contraction.
This proves that $F(j_1-2)\ne0$.

We apply Property \eqref{it.3} to $i_1=j_1-1$. By above,
$(j_1-1)E(j_1)
 \left\{\begin{matrix}
j_1 \; j_2 \; j_3 \\
\;\;j_4 \;\; \; j_5 \; j_6
\end{matrix}
\right\}\neq 0
$ and, by hypothesis,
$
 \left\{\begin{matrix}
j_1\!-\!1 \; j_2 \; j_3 \\
\;\;j_4 \;\; \; j_5 \; j_6
\end{matrix}
\right\}= 0
$. We infer
$  \left\{\begin{matrix}
j_1\!-\!2 \; j_2 \; j_3 \\
\;\;j_4 \;\; \; j_5 \; j_6
\end{matrix}
\right\}\ne 0$.

We finally apply Property \eqref{it.3}
to $i_1=j_1-2$. By hypothesis,
$
 \left\{\begin{matrix}
j_1\!-\!1 \; j_2 \; j_3 \\
\;\;j_4 \;\; \; j_5 \; j_6
\end{matrix}
\right\}= 0
$, while $F(j_1-2)\left\{\begin{matrix}
j_1\!-\!2 \; j_2 \; j_3 \\
\;\;j_4 \;\; \; j_5 \; j_6
\end{matrix}
\right\}\ne 0$. It follows that
$\left\{\begin{matrix}
j_1\!-\!3 \; j_2 \; j_3 \\
\;\;j_4 \;\; \; j_5 \; j_6
\end{matrix}
\right\}\ne0$.
\end{proof}

\begin{remark}
 Proposition \ref{prop.6j_non-zero} is not true without the hypothesis $j_2=j_3$
and, indeed, there are many examples showing this.
 For instance, if $(j_1,j_2,j_3,j_4,j_5,j_6)$ is either
\[
 (3, 3, 2, 2, 1, 2),\;(4,3/2,7/2,3/2,3,1),\;(6,5/2,13/2,3,9/2,3/2)
\]then
$j_1=j_5+j_6\ge3$, the triples \eqref{eq.4triples} satisfy the triangle
condition for  $h=j_1$, $h=j_1-1$;
$\left\{\begin{matrix}
j_1\!-\!1 \; j_2 \; j_3 \\
\;\;j_4 \;\; \; j_5 \; j_6
\end{matrix}
\right\}
=0$
(this can be verified with an on-line calculator or
from the explicit formula for the $6j$-symbol given in \cite[\S9.2.1]{VMK})
but $(j_1-3,j_2,j_3)$ does not satisfy the triangle condition and thus
$\left\{\begin{matrix}
j_1\!-\!3 \; j_2 \; j_3 \\
\;\;j_4 \;\; \; j_5 \; j_6
\end{matrix}
\right\}=0$.
\end{remark}

\section{Uniqueness of faithful uniserial \texorpdfstring{$\g$}{}-modules of length 3}

\begin{theorem}\label{uno}
Suppose $V$ is a faithful uniserial $\g$-module of length 3 with socle factors $V(a),V(b),V(c)$.
Then $c=a$. Moreover,
\begin{enumerate}[(i)]
 \item If $m=1$ then either $b=a+1$, or $a\geq 1$ and $b=a-1$.
 \item If $m\geq 3$ then either $a=0$ and $b=m$, or $a=1$ and $b=m+1$, or $a=1$ and $b=m-1$, or $m=3$, $a=4$ and $b=3$.
\end{enumerate}
Furthermore, in all cases the isomorphism type of $V$ is completely determined by that of its socle factors.
\end{theorem}

\begin{proof} Let $0=V_0\subset V_1\subset V_2\subset V_3=V$ be the only composition series of $V$
as $\g$-module. As $\sel(2)$ is semisimple, there exist $\sel(2)$-submodules
$W_2$ and $W_3$ of $V$ such that
$V_2=V_1\oplus W_2$ and $V_3=V_2\oplus W_3$.
Here $V_1\simeq V(a)$, $W_2\simeq V(b)$ and $W_3\simeq V(c)$. Let $B_1,B_2,B_3$
be bases of $V_1,W_2,W_3$, respectively, so that $B=B_1\cup B_2\cup B_3$ is adapted to
a basis of $V$.
Since $\h_n=[\g,\h_n]$, it follows from \S\ref{sec.matrix_recognition}
that there is a block upper triangular matrix
representation $R:\g\to\gl(d)$, $d=a+b+c+3$, relative to $B$, of the form
\begin{equation}
\label{zeta}
R(s+h)=\left(
         \begin{array}{ccc}
           R_1(s) & X(h) & Z(h) \\
           0 & R_2(s) & Y(h) \\
           0 & 0 & R_3(s) \\
         \end{array}
       \right),\quad s\in\sl(2),\; h\in\h_n.
\end{equation}

We may view $\gl(d)$ as a $\g$-module via $x\cdot A=[R(x),A]$.
Note that the 6 upper triangular blocks, say $M_{11}, M_{22}, M_{33}, M_{12}, M_{23}, M_{13}$,
become $\sel(2)$-submodules of $\gl(d)$. Moreover,
$X:\h_n\to M_{12}$, $Y:\h_n\to M_{23}$, $Z:\h_n\to M_{13}$ are $\sel(2)$-homomorphisms
and $M_{12}\simeq\Hom(V(b),V(a))$, $M_{23}\simeq\Hom(V(c),V(b))$, $M_{13}\simeq\Hom(V(c),V(a))$ as $\sel(2)$-modules.
By (\ref{zeta}), $R(\z)=R([\h_n,\h_n])\subseteq M_{13}$.
Thus, $X$ and $Y$ vanish on $\z$. Since $R$ is faithful, $Z$ does
not vanish on $\z$.
Hence $V(0)$ enters  $V(c)\otimes V(a)$
and this implies, by the Clebsch-Gordan formula, that
$c=a$ and $Z(\z)$ consists of scalar operators. Moreover, since $m\not\equiv 2a\mod 2$, $Z$ must vanish on $V(m)$
and therefore $Z$ is completely determined by $X$ and $Y$,
whose restrictions to $V(m)$ are non-trivial by Theorem \ref{thm.CS_2.4}.
Conjugating by
a suitable block diagonal matrix, with each block a scalar matrix,
we can arbitrarily scale all blocks in the first superdiagonal.
This  shows that $V$ is uniquely determined by its socle factors (cf. \cite[Proposition 3.2]{CS}).
Furthermore, since $V(m)$
enters  $V(a)\otimes V(b)$, we obtain (i).

We assume for the remainder of the proof that $m\geq 3$. Consider the $\sl(2)$-homomorphism $\Lambda^2 V(m)\to M_{13}$ given by
$u\wedge v=X(u)Y(v)-X(v)Y(u)=Z[u,v]$. By above, its image, say $\mathcal J$, is isomorphic to $V(0)$. Set
$r=\min\{2m-2,2a\}$ if $a$ is even, and $r=\min\{2m-2,2a-2\}$ if $a$ is odd.

Suppose first
\begin{equation}
\label{6j}
\left\{\begin{matrix}
\frac{m}{2} \; \frac{r}{2} \; \frac{m}{2}  \\[.6mm]
\frac{a}{2} \;\, \frac{b}{2} \;\, \frac{a}{2}
\end{matrix}
\right\}
\end{equation}
is non-zero. Then, according to \cite[Corollary 9.2]{CS}, $V(r)$ enters  $\mathcal J$.
Therefore $r=0$.
Recalling that $m\geq 3$, and taking into account that $V(m)$
enters  $V(a)\otimes V(b)$, we obtain $a=0$ with $b=m$
if $a$ is even, and $a=1$ with $b=m \pm 1$ if $a$ is odd.

Suppose next (\ref{6j}) is zero. We deal first with the case when $a$ is even. If $r=2a$ then
all four triples (\ref{triples}) taken from (\ref{6j}) satisfy the triangle condition and $(\frac{a}{2},\frac{r}{2},\frac{a}{2})$ is a degenerate triangle,
so Property \eqref{it.2} yields that (\ref{6j}) is non-zero, a contradiction. Therefore, we must have $r=2m-2$ with $m-1<a$.
Set
\begin{align*}
 j_1&=m,& j_2=j_3=\frac{a}{2},\\
j_4&=\frac{b}{2},& j_5=j_6=\frac{m}{2}.
\end{align*}
From the fact that (\ref{6j}) is zero, it follows from Property \eqref{it.4} that
\begin{equation}
\label{6j2}
\left\{\begin{matrix}
j_1\!-\!1 \; j_2 \; j_3  \\[.6mm]
\;\; j_4 \;\;\; j_5 \; j_6
\end{matrix}
\right\}=0.
\end{equation}
Moreover, all four triples (\ref{triples}) taken from (\ref{6j2}) satisfy the triangle condition. Furthermore, since $m\leq a$,
all four triples (\ref{triples}) taken from $(j_1,j_2,j_3,j_4,j_5,j_6)$ satisfy the triangle condition. Thus, all hypotheses of
Proposition \ref{prop.6j_non-zero} are met. We obtain
$\left\{\begin{matrix}
j_1\!-\!3 \; j_2 \; j_3 \\
\;\;j_4 \;\; \; j_5 \; j_6
\end{matrix}
\right\}\ne0.
$
Making use of \cite[Corollary 9.2]{CS} and Property \eqref{it.4}, we infer that
$V(r-4)$ appears in $\mathcal J$. Thus $r=4$, that is, $m=3$.
We now need to find out the possible values of $a$ and $b$.

The fact that (\ref{6j}) is zero becomes
\begin{equation}
\label{te0}
\left\{\begin{matrix}
\frac{4}{2} \; \frac{3}{2} \; \frac{3}{2}  \\[.6mm]
\frac{b}{2} \; \frac{a}{2} \; \frac{a}{2}
\end{matrix}
\right\}=\left\{\begin{matrix}
\frac{3}{2} \; \frac{4}{2} \; \frac{3}{2}  \\[.6mm]
\frac{a}{2} \; \frac{b}{2} \; \frac{a}{2}
\end{matrix}
\right\}=0.
\end{equation}
Now
\begin{equation}
\label{te}
\left\{\begin{matrix}
\frac{4}{2}\!+\!{\scriptstyle 1} \; \frac{3}{2} \; \frac{3}{2}  \\[.6mm]
\;\;\;\frac{b}{2} \;\;\; \frac{a}{2} \; \frac{a}{2}
\end{matrix}
\right\}
\ne0,
\end{equation}
since $a\geq 3$ and therefore all four triples (\ref{triples}) taken from (\ref{te}) satisfy the triangle condition
and we have the degenerate triangle $(\frac{6}{2},\frac{3}{2},\frac{3}{2})$. On the other hand,
\begin{equation}
\label{te1}
\left\{\begin{matrix}
\frac{4}{2}\!+\!{\scriptstyle 2} \; \frac{3}{2} \; \frac{3}{2}  \\[.6mm]
\;\;\;\frac{b}{2}\;\;\; \frac{a}{2} \; \frac{a}{2}
\end{matrix}
\right\}
=0
\end{equation}
as $(\frac{8}{2},\frac{3}{2},\frac{3}{2})$ does not satisfy the triangle condition. It follows from Property \eqref{it.3} applied to
$(i_2,i_3,i_4,i_5,i_6)=(\frac{3}{2},\frac{3}{2},\frac{b}{2},\frac{a}{2},\frac{a}{2})$ that $F(3)=0$. The definition of $F(3)$ now gives
\begin{equation}\label{eq.ab}
a(a+2)=b(b+2)+9,
\end{equation}
and the only pair $(a,b)$ of non-negative integers satisfying \eqref{eq.ab} is $(a,b)=(4,3)$.

The final case, when $a$ is odd, is impossible. Indeed, set
\begin{align*}
 j_1&=\frac{r}{2}+1=\min\{m,a\},& j_2=j_3=\max\left\{\frac{m}{2},\frac{a}{2}\right\},\\
j_4&=\frac{b}{2},& j_5=j_6=\min\left\{\frac{m}{2},\frac{a}{2}\right\}.
\end{align*}
Then Proposition \ref{prop.6j_non-zero} applies to give $r=4$.
As above, this gives $m=4$ or $a=4$, which is impossible
since $m$ and $a$ are both odd.
\end{proof}

\section{Construction of faithful uniserial \texorpdfstring{$\g$}{}-modules of length 3}\label{sec.construction}

\begin{theorem}\label{dos}
In all cases below there is faithful uniserial $\g$-module of length
3 with socle factors $V(a),V(b),V(a)$.
\begin{enumerate}[(i)]
 \item $m=1$ with $b=a+1$ or $b=a-1$ (in the latter case $a>0$).
 \item$m\geq 3$, with $(a,b)=(0,m)$ or $(a,b)=(1,m+1)$ or $(a,b)=(1,m-1)$ or $(m,a,b)=(3,4,3)$.
\end{enumerate}
\end{theorem}

\begin{proof} We will give an explicit faithful uniserial representation
$R:\g\to\gl(d)$, $d=2a+b+3$, in every case listed above.
In each case,
$$
R(s+v+a z)=\left(
         \begin{array}{ccc}
           R_a(s) & X(v) & Z(a z) \\
           0 & R_b(s) & Y(v) \\
           0 & 0 & R_a(s) \\
         \end{array}
       \right),\quad s\in\sl(2),\; v\in V(m),\; a\in\C.
$$
Here $X:V(m)\to\Hom(V(b),V(a))$ and $Y:V(m)\to\Hom(V(a),V(b))$
are $\sl(2)$-homomorphisms given in matrix form relative to standard bases of $V(a)$ and $V(b)$, and
$R_a$ and $R_b$ are as given in \S\ref{lieg}. Moreover, $Z:\z\to\gl(V(a))$
is an $\sl(2)$-homomorphism given in matrix scalar form.
It is straightforward to see (cf. \cite[\S 3]{CS})
that such $R$ is indeed a Lie homomorphism
(and hence a faithful uniserial representation by Theorem \ref{thm.CS_2.4}) provided $Z\ne0$ and
\begin{equation}
\label{funca}
X(v_i)Y(v_j)-X(v_j)Y(v_i)=Z([v_i,v_j]),\quad 0\leq i\leq m.
\end{equation}
We leave it to the reader to verify (\ref{funca}) in each case, recalling from \S\ref{lieg} that
$$
[v_i,v_j]=0\text{ if }i+j\neq m,\text{ while } [v_i,v_{m-i}]=(-1)^i\binom{m}{i}z.
$$
Let $A'$ stand for the transpose of a matrix $A$ and set
$v=\underset{0\leq i\leq m}\sum a_iv_i\in V(m)$, $a_i\in\C$.

\medskip

\begin{enumerate}[(1)]
 \item\label{it.0m} $m\geq 1$ and $V$ has socle factors $V(0),V(m),V(0)$.

\medskip
\noindent
 $Z(z)=(2)$.

\medskip
\noindent
\[
X(v)=\begin{pmatrix}
    -\binom{m}{m}a_{m} & \binom{m}{m-1}a_{m-1} & \cdots & \binom{m}{2}a_{2} & -\binom{m}{1}a_{1} & \binom{m}{0}a_{0}
\end{pmatrix}.
\]
\medskip
\noindent
\[
Y(v)=
\begin{pmatrix}
     a_{0} & a_{1}  & \cdots & a_{m-1} &  a_{m}
\end{pmatrix}'.
\]

\item\label{it.1m} $m\geq 1$ and $V$ has socle factors $V(1),V(m+1),V(1)$.

\medskip
\noindent

$Z(z)=\begin{pmatrix}
  m+2 & 0 \\ 0 & m+2
 \end{pmatrix}
$.

\medskip
\noindent

\[
X(v)=\begin{pmatrix}
    -\binom{m}{m}a_{m} & \binom{m}{m-1}a_{m-1} & \cdots & \binom{m}{2}a_{2} & -\binom{m}{1}a_{1} & \binom{m}{0}a_{0} & 0 \\[2mm]
    0 & -\binom{m}{m}a_{m} & \binom{m}{m-1}a_{m-1}  & \cdots & \binom{m}{2}a_{2} & -\binom{m}{1}a_{1} & \binom{m}{0}a_{0}
\end{pmatrix}.
\]
\medskip
\noindent

\[
Y(v)=\begin{pmatrix}
    (m+1) a_{0} & ma_{1}  & \cdots & 2a_{m-1} &  a_{m} & 0 \\[2mm]
              0 & a_{0} & 2a_{1}   & \cdots & ma_{m-1} &  (m+1)a_{m}
\end{pmatrix}'.
\]

\item\label{it.1m-1} $m\geq 1$ and $V$ has socle factors $V(1),V(m-1),V(1)$.

\medskip
\noindent

$Z(z)=I_2$.

\medskip
\noindent

\[
X(v)=\begin{pmatrix}
 \binom{m-1}{m-1}a_{m-1} & -\binom{m-1}{m-2}a_{m-2} & \cdots & \binom{m-1}{2}a_{2} & -\binom{m-1}{1}a_{1} & \binom{m-1}{0}a_{0} \\[2mm]
 \binom{m-1}{m-1}a_{m} & -\binom{m-1}{m-2}a_{m-1}  & \binom{m-1}{m-3}a_{m-2}& \cdots & -\binom{m-1}{1}a_{2} & \binom{m-1}{0}a_{1}
\end{pmatrix}.
\]
\medskip
\noindent

\[
Y(v)=\begin{pmatrix}
     a_{1}  & a_{2}  & \cdots & a_{m-1} &  a_{m}  \\[2mm]
     -a_{0} & -a_{1} & -a_{2}   & \cdots & -a_{m-1}
\end{pmatrix}'.
\]

\item\label{it.a} $m=1$ and $V$ has socle factors $V(a)$, $V(a+1)$, $V(a)$.

Let $I_k$ be the identity matrix of size $k$,
let $0_k$ be the zero column matrix with $k$ rows.
Let $J^+_k,J^-_k$ be the diagonal matrices of size $k$ given by
\[
J_k^+=
\begin{pmatrix}
    k & 0 & \cdots & 0 \\
    0 & k-1 & \cdots & 0 \\[-1mm]
    0 & 0 & \ddots & 0 \\
    0 & 0 & \cdots & 1
    \end{pmatrix},\qquad
J_k^-=
\begin{pmatrix}
    1 & 0 & \cdots & 0 \\
    0 & 2 & \cdots & 0 \\[-1mm]
    0 & 0 & \ddots & 0 \\
    0 & 0 & \cdots & k
    \end{pmatrix}.
\]

\medskip
\noindent
$Z(z)=(a+2)I_{a+1}$.

\medskip
\noindent
\[
X(v_0)=\begin{pmatrix}
    0_{a+1} & I_{a+1}
\end{pmatrix}\quad\text{ and }\quad
X(v_1)=\begin{pmatrix}
    -I_{a+1} & 0_{a+1}
\end{pmatrix}.
\]
\medskip
\noindent
\[
Y(v_0)=
\begin{pmatrix}
    J^+_{a+1} \\[3mm] 0_{a+1}'
\end{pmatrix}\quad\text{ and }\quad
Y(v_1)=\begin{pmatrix}
    0_{a+1}' \\[3mm] J^-_{a+1}
\end{pmatrix}.
\]

\item\label{it.a1} $m=1$ and $V$ has socle factors $V(a+1)$, $V(a)$, $V(a+1)$.

\medskip
\noindent
$Z(z)=-(a+1)I_{a+2}$.

\medskip
\noindent
\[
X(v_0)=\begin{pmatrix}
    J^+_{a+1} \\[3mm] 0_{a+1}'
\end{pmatrix}\quad\text{ and }\quad
X(v_1)=\begin{pmatrix}
    0_{a+1}' \\[3mm] J^-_{a+1}
\end{pmatrix}.
\]
\medskip
\noindent
\[
Y(v_0)=\begin{pmatrix}
    0_{a+1} & I_{a+1}
\end{pmatrix}\quad\text{ and }\quad
Y(v_1)=\begin{pmatrix}
    -I_{a+1} & 0_{a+1}
\end{pmatrix}.
\]

\item\label{it.434} $m=3$ and $V$ has socle factors $V(4)$, $V(3)$, $V(4)$.

\medskip
\noindent
$Z(z)=6I_{5}$.

\medskip
\noindent
The matrices $X(v_0),X(v_1),X(v_2),X(v_3)\in V(m)$  are respectively:
{\footnotesize
\[
\begin{pmatrix} 0&6&0&0\\ 0&0&3&0\\ 0&0&0&1\\ 0&0&0&0\\ 0&0&0&0\end{pmatrix}, \quad
\begin{pmatrix} -6&0&0&0\\ 0&0&0&0\\ 0&0&3&0\\ 0&0&0&3\\ 0&0&0&0\end{pmatrix}, \quad
\begin{pmatrix} 0&0&0&0\\ -3&0&0&0\\ 0&-3&0&0\\ 0&0&0&0\\ 0&0&0&6\end{pmatrix}, \quad
\begin{pmatrix} 0&0&0&0\\ 0&0&0&0\\ -1&0&0&0\\ 0&-3&0&0\\ 0&0&-6&0\end{pmatrix}.
\]}

\medskip
\noindent
The matrices $Y(v_0),Y(v_1),Y(v_2),Y(v_3)\in V(m)$  are respectively:
{\footnotesize
\[
 \begin{pmatrix} 0&0&3&0&0\\ 0&0&0&2&0\\ 0&0&0&0&1\\ 0&0&0&0&0\end{pmatrix}, \quad
 \begin{pmatrix} 0&-6&0&0&0\\ 0&0&-3&0&0\\ 0&0&0&0&0\\ 0&0&0&0&3\end{pmatrix}, \quad
\begin{pmatrix} 3&0&0&0&0\\ 0&0&0&0&0\\ 0&0&-3&0&0\\ 0&0&0&-6&0\end{pmatrix}, \quad
\begin{pmatrix} 0&0&0&0&0\\ 1&0&0&0&0\\ 0&2&0&0&0\\ 0&0&3&0&0\end{pmatrix}.
\]}
\end{enumerate}

\end{proof}

\section{Non-existence of faithful uniserial \texorpdfstring{$\g$}{}-modules of length \texorpdfstring{$\geq 4$}{}}

\begin{prop}\label{no4} The are no faithful uniserial $\g$-modules of length 4.
\end{prop}

\begin{proof} Suppose, if possible, that $V$ is a faithful uniserial $\g$-module of length 4,
with socle factors
$V(a),V(b),V(c),V(d)$. Then $V$ has a uniserial submodule $W_1$ with socle factors $V(a),V(b),V(c)$
and a uniserial
quotient $W_2=V/V(a)$ with socle factors $V(b),V(c),V(d)$. Four cases arise, depending on whether
$W_1,W_2$ are faithful
or not.
The case when $W_2$ is faithful, but $W_1$ is not, follows by duality (see \cite[Lemma 2.6]{CS})
from the case when
$W_1$ is faithful but $W_2$ is not.

\medskip

\noindent{\sc Case 1.} $W_1$ is faithful and $m\geq 3$.
By Theorem \ref{uno}, $(a,b,c)$ must be one of
$(0,m,0),(1,m+1,1),(1,m-1,1),(4,3,4)$, where in the latter case $m=3$. If $W_2$ is faithful,
a second application of Theorem \ref{uno},
this time to $(b,c,d)$, leaves no possible value for $b$. If $W_2$ is not faithful we appeal to the
classification of uniserial $\sl(2)\ltimes V(m)$-modules of length 3 given in
Theorem \ref{thm.CS_Classification}.
It forces $c=m$ to be in $\{0,1,4\}$, or $(b,c,d)$ to be an
arithmetic progression of step $\pm m$, which is impossible.

\medskip

\noindent{\sc Case 2.} $W_1$ is faithful and $m=1$.

Suppose first $W_2$ is also faithful. From Theorem \ref{uno} we deduce
that $(a,b,c,d)=(a,a+1,a,a+1)$ or $(a,b,c,d)=(b+1,b,b+1,b)$.

Let us deal with the case $(a,b,c,d)=(a,a+1,a,a+1)$ first.
Consider a basis $B=B_1\cup B_2\cup B_3\cup B_4$ of
$V$ adapted to the composition (socle) series.
We choose $B_i$ so that the matrix representation $R:\g\to\gl(d)$, $d=4a+6$, relative to $B$ has the form
\begin{equation}
\label{rept}
R(s+h)=\left(
  \begin{array}{cccc}
    R_a(s) & A(h) & D(h) & F(h) \\
    0 & R_{a+1}(s) & B(h) & E(h) \\
    0 & 0 & R_a(s) & C(h) \\
    0 & 0 & 0 & R_{a+1}(s) \\
  \end{array}
\right),\quad s\in\sl(2),\;h\in\h_n.
\end{equation}
Here $A,C,F:\h_n\to\Hom(V(a+1),V(a))$, $B:\h_n\to\Hom(V(a),V(a+1))$, $D:\h_n\to\Hom(V(a),V(a))$ and
$E:\h_n\to\Hom(V(a+1),V(a+1))$ are $\sl(2)$-homomorphisms given in matrix form and are unique up to scaling.
Therefore, since both $W_1$ and $W_2$ are faithful and uniserial (and $B$ is part of the definition of both modules),
it follows from Theorem \ref{uno} that $A,B,C,D,E$ are exactly as given in \S\ref{sec.construction}.
In particular $D(z)=(a+2)I_{a+1}$ and $E(z)=-(a+1)I_{a+2}$.
Now $R([v_1,z])=0$, whereas block (1,4) of $[R(v_1),R(z)]$ is $-(2a+3)
\begin{pmatrix}
    0_{a+1} & I_{a+1}
\end{pmatrix},$
a contradiction. This shows that this case is impossible for all $a$.

It is easy to see that the case $(a,b,c,d)=(b+1,b,b+1,b)$ is dual to the previous one
(see also \cite[Lemma 2.6]{CS}) and is therefore impossible.

Suppose next $W_2$ is not faithful. Appealing to
Theorem \ref{thm.CS_Classification} we deduce that
$(a,b,c,d)=(a,a+1,a,a-1)$ or $(a,b,c,d)=(b+1,b,b+1,b+2)$.

Arguing as above we find that block (1,4) of $[R(v_1),R(z)]$ is $-(a+2)\begin{pmatrix}
    J^+_{a} \\[3mm] 0_{a}'
\end{pmatrix}$ in the first case and
$(b+1)\begin{pmatrix}
    0_{b+2} & I_{b+2}
\end{pmatrix}$ in the second one. Both cases are impossible.

\medskip

\noindent{\sc Case 3.} Neither $W_1$ nor $W_2$ is faithful. Since $V$ is faithful, then $V(0)$
must enter $\Hom(V(d),V(a))$, whence $d=a$.
We appeal, again, to the classification of uniserial $\sl(2)\ltimes V(m)$-modules of length 3
given in Theorem \ref{thm.CS_Classification}. Since $d=a$,
we see that $(a,b,c)$ cannot be in an arithmetic progression
of step $\pm m$,
which forces $(a,b,c)$ to be $(0,m,c)$ or $(a,m,0)$.
In the latter case $(b,c,d)=(m,0,a)$, against
Theorem \ref{thm.CS_Classification}.
In the former, $(b,c,d)=(m,m,0)$ by
Theorem \ref{thm.CS_Classification}.
But this is only possible when $m\equiv 2m\mod 4$, which is not the case.
\end{proof}

\begin{theorem}\label{no5} The are no faithful uniserial $\g$-modules of length $\ell >3$.
\end{theorem}

\begin{proof} By induction on $\ell$. The base case $\ell=4$ is proven in Proposition \ref{no4}. Suppose
$V$ is a uniserial $\g$-module of length $\ell>4$ and there are no faithful uniserial $\g$-modules of length $\ell-1$. Let
$V(a_1),\dots,V(a_\ell)$ be the socle factors of $V$.
Then $V$ has a submodule
$W_1$ with socle factors $V(a_1),\dots,V(a_{\ell-1})$ and a quotient module $V/V(a_1)$ with socle factors
$V(a_2),\dots,V(a_\ell)$.
By inductive hypothesis, these uniserial $\g$-modules are not faithful.
Therefore, they are uniserial $\sl(2)\ltimes V(m)$-modules.
The classification of uniserial $\sl(2)\ltimes V(m)$-modules
of length $\geq 4$
given in Theorem \ref{thm.CS_Classification} forces $a_1,\dots,a_\ell$ to be an arithmetic progression of step $\pm m$.
Thus, $V(0)$ does not enter $\Hom(V(a_j),V(a_i)))$ for any $i<j$, so $\z$ acts trivially on $V$.
\end{proof}


\end{document}